\documentclass[12pt,reqno]{article}
\usepackage{graphicx}
\usepackage{amsmath}
\usepackage{amsthm}
\usepackage{latexsym,bm}
\usepackage{amssymb}
\usepackage{indentfirst}
\newtheorem{thm}{Theorem}[section]
\newtheorem{cor}[thm]{Corollary}
\newtheorem{lem}[thm]{Lemma}
\newtheorem{prop}[thm]{Proposition}
\newtheorem{rem}[thm]{Remark}

\numberwithin{equation}{section}

\begin{document}

\title{{\bf A characterization of round spheres in space forms}}

\author{Francisco Fontenele\thanks{Partially supported by CNPq (Brazil)}\, and Roberto Alonso N\'u\~nez}

\date{}
\maketitle

\footnotetext[1]{2010 {\it Mathematics Subject Classication.}
Primary 53C42, 14J70; Secondary 53C40, 53A10.}


\begin{quote}
\small{\bf Abstract.} Let $\mathbb Q^{n+1}_c$ be the complete
simply-connected $(n+1)$-dimensio-nal space form of
curvature $c$. In this paper we obtain a new
characterization of geodesic spheres in $\mathbb Q^{n+1}_c$ in terms of the higher order mean curvatures. In
particular, we prove that the geodesic sphere is the only complete
bounded immersed hypersurface in $\mathbb Q^{n+1}_c,\;c\leq 0,$
with constant mean curvature and constant scalar curvature. The
proof relies on the well known Omori-Yau maximum principle, a
formula of Walter for the Laplacian of the $r$-th mean curvature
of a hypersurface in a space form, and a classical inequality of
G\aa rding for hyperbolic polynomials.
\end{quote}

\section{Introduction }

A question of interest in differential geometry is whether the
geodesic sphere is the only compact oriented hypersurface in the
$(n+1)$-dimensional Euclidean space $\mathbb R^{n+1}$ with
constant $r$-th mean curvature $H_r$, for some $r=1,...,n$ ($H_1$,
$H_2$, and $H_n$ are the mean curvature, the scalar curvature, and
the Gauss-Kronecker curvature, respectively -- see the definitions
in Section 2). When $r=1$ this question is the well known Hopf
conjecture, and when $r=2$ it is a problem proposed by Yau
\cite[Problema 31, p. 677]{Y2}.

As proved by Alexandrov \cite{Al} for $r=1$, and by Ros
\cite{Ro1,Ro2} (see also \cite{Kor,MR}) for any $r$, the above
question has an affirmative answer for embedded hypersurfaces. In
the immersed case, the question has a negative answer when $r=1$
(by the examples of non-spherical compact hypersurfaces with
constant mean curvature in the Euclidean space constructed by
Wente \cite{We} and by Hsiang, Teng and Yu \cite{HTY}), and an
affirmative answer when $r=n$ (by a theorem of Hadamard). The
problem is still unsolved for $1<r<n$. For partial answers when
$r=2$ (Yau's problem), see \cite{Ch,Li,Ok}.

Because of the difficulty of the above question, it is natural to attempt to obtain the rigidity of the sphere in $\mathbb R^{n+1}$ under geometric conditions stronger than $H_r$ be constant for some $r$. In this regard, Gardner \cite{Gar} proved that if a compact oriented hypersurface $M^n$ in $\mathbb R^{n+1}$ has
two consecutive mean curvatures $H_r$ and $H_{r+1}$ constant, for
some $r=1,...,n-1$, then it is a geodesic sphere. For generalizations of this result see \cite{Bi,Ko,Wan}.

In \cite{CW}, Cheng and Wan proved that a complete hypersurface
$M^3$ with constant scalar curvature $R$ and constant mean
curvature $H\neq 0$ in $\mathbb R^4$ is a generalized cylinder
$\mathbb S^k(a)\times\mathbb R^{3-k}$, for some $k=1,2,3$ and some
$a>0$ (see \cite{Nu} for results of this nature in higher
dimensions). From this result one obtains the following
improvement, when $n=3$ and $r=1$, in the theorem of Gardner
referred to above: {\it The geodesic spheres are the only complete
bounded immersed hypersurfaces in $\mathbb R^4$ with constant scalar curvature and constant mean curvature} (compare with Corollary
\ref{Cor1}).

Our main result (Theorem \ref{MainResult}) provides a new
characterization of geodesic spheres in space forms. There are
many results of this nature in the literature, most of which
assuring that a compact hypersurface that satisfies certain
geometric conditions is a geodesic sphere. What makes special the
characterization provided by Theorem \ref{MainResult} is that in
it the geometric conditions are imposed on a {\it complete}
hypersurface (that is bounded when $c\leq 0$, and contained in a
spherical cap when $c>0$), and not on a {\it compact} one.

In the theorem below, as well as in the remaining of this work,
$\mathbb Q^{n+1}_c$ stands for the $(n+1)$-dimensional complete
simply-connected space of constant sectional curvature $c$.

\begin{thm}\label{MainResult}
Let $M^n$ be a complete Riemannian manifold with scalar curvature
$R$ bounded from below, and let $f:M^n\to\mathbb Q^{n+1}_c$ be an
isometric immersion. In the case $c\leq 0$, assume that $f(M^n)$
is bounded, and in the case $c>0$, that $f(M^n)$ lies inside a
geodesic ball of radius $\rho<\pi/2\sqrt{c}$. If the mean
curvature $H$ is constant and, for some $r=2,...,n$, the $r$-th
mean curvature $H_r$ is constant, then $f(M^n)$ is a geodesic
sphere of $\mathbb Q^{n+1}_c$.
\end{thm}

The following results follow immediately from the above theorem.
Notice that the hypothesis in Theorem \ref{MainResult} that the
scalar curvature of $M^n$ is bounded from below is superfluous
when $r=2$.

\begin{cor}\label{Cor1}
Let $f:M^n\to\mathbb Q^{n+1}_c$ be an isometric immersion of a
complete Riemannian manifold $M^n$ in $\mathbb Q^{n+1}_c$. In the
case $c\leq 0$, assume that $f(M^n)$ is bounded, and in the case
$c>0$, that $f(M^n)$ lies inside a geodesic ball of radius
$\rho<\pi/2\sqrt{c}$. If the mean curvature $H$ and the scalar
curvature $R$ are constant, then $f(M^n)$ is a geodesic sphere of
$\mathbb Q^{n+1}_c$.
\end{cor}

\begin{cor}\label{Cor2}
Let $f:M^n\to\mathbb Q^{n+1}_c$ be an isometric immersion of a
compact Riemannian manifold $M^n$ in $\mathbb Q^{n+1}_c$. In the
case $c>0$, assume that $f(M)$ is contained in an open hemisphere
of  $\mathbb S^{n+1}_c$. If the mean curvature $H$ is constant
and, for some $r=2,...,n$, the $r$-th mean curvature $H_r$ is
constant, then $f(M^n)$ is a geodesic sphere of $\mathbb
Q^{n+1}_c$.
\end{cor}

\begin{rem}{\em The examples of Wente \cite{We} and Hsiang, Teng and Yu
\cite{HTY}, referred to in the second paragraph of this section,
show that the hypothesis that $H_r$ is constant for some
$r,\;2\leq r\leq n$, can not be removed from Theorem
\ref{MainResult}. It is surely a difficult question to know
whether the theorem holds without the assumption that $H$ is
constant (cf. Yau's problem mentioned in the beginning of this
section). We do not know whether Theorem \ref{MainResult} (for
$r\geq 3$) holds without the hypothesis that the scalar curvature
of $M$ is bounded below.}
\end{rem}

The proof of Theorem \ref{MainResult} relies on the well known
Omori-Yau maximum principle \cite{CY1,Om,Y1}, a formula of Walter
\cite{Wa} for the Laplacian of the $r$-th mean curvature of a
hypersurface in a space form, and a classical inequality of G\aa
rding \cite{Gaa} for hyperbolic polynomials.

\section{Preliminaries}

Given an isometric immersion $f:M^n\to N^{n+k}$ of a
$n$-dimensional Riemannian manifold $M^n$ into a
$(n+k)$-dimensional Riemannian manifold $N^{n+k}$, denote by
$\sigma:TM\times TM\to TM^{\perp}$ the (vector valued) second
fundamental form of $f$, and by $A_{\xi}$ the shape operator of
the immersion with respect to a (locally defined) unit normal
vector field $\xi$. From the Gauss formula one obtains, for all
smooth vector fields $X$ and $Y$,
\begin{eqnarray}\label{Sigma}
\langle A_{\xi}X,Y\rangle=\langle\sigma(X,Y),\xi\rangle.
\end{eqnarray}

In the particular case that $M$ and $N$ are orientable and $k=1$,
one may choose a global unit normal vector field $\xi$ and so
define a (symmetric) $2$-tensor field $h$ on $M$ by
$h(X,Y)=\langle\sigma(X,Y),\xi\rangle$. Then, by (\ref{Sigma}),
\begin{eqnarray}\label{SecondFundForm} h(X,Y)=\langle
AX,Y\rangle,\;\;\;X,Y\in\mathfrak X(M),
\end{eqnarray}
where $A=A_{\xi}$ is the shape operator of the immersion with
respect to $\xi$. If we assume further that $N^{n+1}$ has constant
sectional curvature, it follows from the symmetry of $h$ and the
Codazzi equation that the covariant derivative $\nabla h$ of $h$
is symmetric. Hence, $\nabla^2h:=\nabla(\nabla h)$ is symmetric in
the first three entries. The following lemma shows what happens
when we interchange vectors in its third and forth entries. In its
statement, as well as in the remaining of the work, we denote by
$h_{ij}$, $h_{ijk}$ and $h_{ijkl}$ the components of $h$, $\nabla
h$ and $\nabla^2h$, respectively, in a local orthonormal frame
field $\{e_1,\ldots,e_n\}$, i.e.,
$$
h_{ij}=h(e_i.e_j),\;\;h_{ijk}=\nabla
h(e_i,e_j,e_k),\;\;h_{ijkl}=\nabla^2h(e_i,e_j,e_k,e_l).
$$
\begin{lem}\label{Inversion}
For any local orthonormal frame field $\{e_1,\ldots,e_n\}$ on
$M^n$, we have
\begin{eqnarray}\label{Inversion1}
h_{ijkl}-h_{ijlk}=\sum_m\mathcal R_{klim}h_{mj}+\sum_m\mathcal
R_{kljm}h_{im},
\end{eqnarray}
for all $i,j,k,l\in\{1,...,n\}$, where $\mathcal R$ is the
Riemannian curvature tensor of $M^n$ and, for example, $\mathcal
R_{klim}=\langle\mathcal R(e_k,e_l)e_i,e_m\rangle$.
\end{lem}

Formula (\ref{Inversion1}) above is well known. For a proof see,
for instance, \cite[p. 1167]{Ch1}.

Given an isometric immersion $f:M^n\to N^{n+1}$, denote by
$\lambda_1,...,\lambda_n$ the principal curvatures of $M^n$ with
respect to a global unit normal vector field $\xi$ (i.e., the
eigenvalues of the shape operator $A=A_{\xi}$). It is well known
that if we label the principal curvatures at each point by the
condition $\lambda_1\leq\cdots\leq\lambda_n$, then the resulting
functions $\lambda_i:M\to\mathbb R,\;i=1,...,n,$ are continuous.

The $r$-th mean curvature $H_r$, $1\leq r\leq n$, of $M^n$
is defined by
\begin{eqnarray}\label{DefinitionHr}
{n \choose
r}H_r=\sum_{i_1<\ldots<i_r}\lambda_{i_1}\ldots\lambda_{i_r}.
\end{eqnarray}
Notice that $H_1$ is the mean curvature $H$ ($=\frac{1}{n}
\text{tr}A$, where $\text{tr}A$ is the trace of $A$) and
$H_n=\lambda_1\lambda_2\ldots\lambda_n$ is the Gauss-Kronecker
curvature of the immersion. In the particular case that $N^{n+1}$
has constant sectional curvature, the function $H_2$ is up to a
constant the (normalized) scalar curvature $R$ of $M^n$. In fact,
if $N^{n+1}$ has constant sectional curvature $c$ and if
$\{e_1,...,e_n\}$ is an orthonormal basis for the tangent space at
a given point of $M^n$ such that $Ae_i=\lambda_ie_i,\;i=1,...,n$,
then the sectional curvature $K(e_i,e_j)$ of the plane spanned by
$e_i$ and $e_j$ is, by the Gauss equation, given by
$$
K(e_i,e_j)=c+\lambda_i\lambda_j,
$$
and so
\begin{eqnarray}\label{SignificadoH2}
R=\frac{1}{{n \choose 2}}\sum_{i<j}K(e_i,e_j)=\frac{1}{{n \choose 2}}\sum_{i<j}(c+\lambda_i\lambda_j)=c+H_2.
\end{eqnarray}

The squared norm $|A|^2$ of the shape
operator $A$ is defined as the trace of $A^2$. It is easy to see that
\begin{eqnarray}\label{Norm2}
|A|^2=\sum_{i}\lambda_i^2.
\end{eqnarray}

From (\ref{DefinitionHr}), (\ref{SignificadoH2}) and (\ref{Norm2})
we obtain the following useful relation involving the mean
curvature $H$, the norm $|A|$ of the shape operator $A$ and the
normalized scalar curvature $R$:
\begin{eqnarray}\label{Relation}
n^2H^2=\left(\sum_{i=1}^n\lambda_i\right)^2=\sum_{i=1}^n\lambda_i^2+\sum_{i\neq
j}\lambda_i\lambda_j=|A|^2+n(n-1)(R-c).
\end{eqnarray}

In terms of the $r$-th symmetric function $\sigma_r:\mathbb
R^n\to\mathbb R$,
\begin{eqnarray}\label{SymmetricFunctionDefinition}
\sigma_r(x_1,\ldots,x_n)=\sum_{i_1<\ldots<i_r}x_{i_1}\ldots
x_{i_r},
\end{eqnarray}
equality (\ref{DefinitionHr}) can be rewritten as
\begin{eqnarray}\label{DefinitionHr2}
{n \choose r}H_r=\sigma_r\circ\overrightarrow\lambda,
\end{eqnarray}
where $\overrightarrow\lambda=(\lambda_1,...,\lambda_n)$ is the
principal curvature vector of the immersion. In order to unify the
notation, we define $H_0=1=\sigma_0$ and $H_r=0=\sigma_r$, for all
$r\geq n+1$.

As one might expect, the knowledge of the properties of the
symmetric functions is very important to the study of the higher
order mean curvatures of a hypersurface. In order to state a
property of the symmetric functions that will be relevant to us,
we will summarize below some of the  results of the classical
article by G\aa rding \cite{Gaa} on hyperbolic polynomials (see
also \cite[p. 268]{CNS} and \cite[p. 217]{FS}).

Let $P:\mathbb R^n\to\mathbb R$ be a homogenous polynomial of
degree $m$ and let $a=(a_1,\ldots,a_n)$ be a fixed vector of
$\mathbb R^n$. We say that $P$ is hyperbolic with respect to the
vector $a$, or in short, that $P$ is $a$-hyperbolic, if for every
$x\in\mathbb R^{n}$ the polynomial in $s$, $P(sa+x)$, has $m$ real
roots. Denote by $\Gamma_P$ the connected component of the set
$\{P\neq 0\}$ that contains $a$. In \cite{Gaa}, G\aa rding
proved that $\Gamma_P$ is an open convex cone, with vertex at the
origin, and that the homogenous polynomial of degree $m-1$ defined
by
\begin{eqnarray}\label{hr12}
Q(x)=\frac{d}{ds}\biggl|_{s=0}P(sa+x)=\sum_{j=1}^na_j\frac{\partial P}{\partial x_j}(x)
\end{eqnarray}
is also $a$-hyperbolic. Moreover, $\Gamma_P\subset\Gamma_Q$.

As can be easily seen, the $n$-th symmetric function $\sigma_n$ is
hyperbolic with respect to the vector $a=(1,\ldots,1)$. Applying
the results of the previous paragraph to $\sigma_n$, and observing
that
\begin{eqnarray}\label{hr13}
\sigma_r(x)=\frac{1}{(n-r)!}\frac{d^{n-r}}{ds^{n-r}}\biggl|_{s=0}\sigma_n(sa+x),\;\;\;r=1,...,n-1,
\end{eqnarray}
one concludes that $\sigma_r,\;1\leq r\leq n,$ is hyperbolic with
respect to $a=(1,...,1)$ and that
\begin{eqnarray}\label{hr14}
\Gamma_1\supset\Gamma_2\supset\ldots\supset\Gamma_n,
\end{eqnarray}
where $\Gamma_r:=\Gamma_{\sigma_r}$.

In \cite{Gaa}, G\aa rding established an inequality for hyperbolic
polynomials involving their completely polarized forms. A
particular case of this inequality, from which the general case is
derived, says that
\begin{eqnarray}\label{DesigualdadeGarding}
\frac{1}{m}\sum_{k=1}^n y_k\frac{\partial P}{\partial x_k}(x) \geq
P(y)^{\frac{1}{m}}P(x)^{1-\frac{1}{m}},\;\;\;\forall
x,y\in\Gamma_P.
\end{eqnarray}
As observed in \cite[p. 269]{CNS}, the above inequality is
equivalent to the assertion that $P^{1/m}$ is a concave function
on $\Gamma_P$. In particular, we have the following result, which
will play an important role in the proof of Theorem
\ref{MainResult}.

\begin{prop}\label{Concave}
For each $r=1,2,...,n$, the function $\sigma_r^{1/r}$ is concave
on $\Gamma_r$.
\end{prop}

\section{The Laplacian of the $r$-th mean curvature}

The symmetric functions $\sigma_r,\;1\leq r\leq n$, defined by
(\ref{SymmetricFunctionDefinition}), arise naturally from the
identity
\begin{eqnarray}\label{SymmetricFunction}
\prod_{s=1}^n(x_s+t)=\sum_{r=0}^n\sigma_r(x)t^{n-r},
\end{eqnarray}
which is valid for all $x=(x_1,...,x_n)\in\mathbb R^n$ and
$t\in\mathbb R$. Differentiating this identity with respect to
$x_j$, one obtains
\begin{eqnarray}\label{SymmetricFunction1}
\prod_{s\neq j}(x_s+t)=\sum_{r=0}^n\frac{\partial\sigma_r}{\partial x_j}(x)t^{n-r},\;\;\;j=1,...,n.
\end{eqnarray}
Differentiation of (\ref{SymmetricFunction1}) with respect to
$x_i$, for $i\neq j$, yields
\begin{eqnarray}\label{SymmetricFunction2}
\prod_{s\neq i,j}(x_s+t)=\sum_{r=0}^n\frac{\partial^2\sigma_r}{\partial x_i\partial x_j}(x)t^{n-r},\;\;\;i\neq j.
\end{eqnarray}
From the identities
\begin{eqnarray}\label{SymmetricFunction3}
\sigma_r(x)=x_i\sigma_{r-1}(\widehat{x_i})+\sigma_r(\widehat{x_i}),\;\;\;x\in\mathbb
R^n,\;\;1\leq i,r \leq n,
\end{eqnarray}
where, for instance,
$\sigma_{r-1}(\widehat{x_i})=\sigma_{r-1}(x_1,...,x_{i-1},x_{i+1},...,x_n)$,
one obtains, for all $r=2,...,n$,
\begin{eqnarray}\label{SymmetricFunction5}
\frac{\partial^2\sigma_r}{\partial x_i\partial x_j}(x)=\begin{cases}\sigma_{r-2}(\widehat{x_i},\widehat{x_j}),
\;\;i\neq j,\\0,\;\;\;\;\;\;\;\;\;\;\;\;\;\;\;\;\;\,i=j.\end{cases}
\end{eqnarray}

In \cite{Wa} Walter established a formula for the Laplacian of the
$r$-th mean curvature of a hypersurface in a space of constant
sectional curvature. For convenience of the reader, we state and
prove that formula below. Recall that the Laplacian $\Delta u$ of
a $C^2$-function $u$ defined on a Riemannian manifold
$(M,\langle,\rangle)$ is the trace of the $2$-tensor field
$\text{Hess}\,u$, called the Hessian of $u$, defined by
$\text{Hess}\,u(X,Y)=\langle\nabla_X\nabla u,Y\rangle$, for all
$X,Y\in\mathfrak X(M)$.

\begin{prop}\label{WalterProposition}
Let $M^n$ be an orientable hypersurface of an orientable
Riemannian manifold $N ^{n+1}_c$ of constant sectional curvature
$c$. Then, for every $r=1,...,n$ and every $p\in M^n$,
\begin{eqnarray}\label{WalterFormula}
{n \choose r}\Delta H_r&=&n\sum_{j}
\frac{\partial\sigma_r}{\partial
x_j}(\overrightarrow{\lambda})\textnormal{Hess}\,H(e_j,e_j)-\sum_{i<j}
\frac{\partial^2\sigma_r}{\partial x_i\partial x_j}(\overrightarrow{\lambda})(\lambda_i-\lambda_j)^2K_{ij}\nonumber\\
&&+\sum_{i,j,k}\frac{\partial^2\sigma_r}{\partial x_i\partial x_j}(\overrightarrow{\lambda})(h_{iik}h_{jjk}-h_{ijk}^2),
\end{eqnarray}
where $\lambda_1,...,\lambda_n$ are the principal curvatures of
$M^n$ at $p$,
$\overrightarrow{\lambda}=(\lambda_1,\ldots,\lambda_n)$,
$\{e_1,...,e_n\}$ is an orthonormal basis of $T_pM$ that
diagonalizes the shape operator $A$, and $K_{ij}$ is the sectional
curvature of $M^n$ in the plane spanned by $\{e_i,e_j\}$.
\end{prop}

\begin{proof}
Extend the orthonormal basis $\{e_1,...,e_n\}$ of $T_pM$ to
a local orthonormal frame field, still denoted by
$\{e_1,...,e_n\}$, through parallel transport of the $e_i$'s along
the geodesics emanating from $p$. From (\ref{SecondFundForm}), one
obtains
\begin{eqnarray}\label{det1}
(A+tI)e_j=\sum_l(h_{lj}+t\delta_{lj})e_l,\;\;\;1\leq j\leq n,\;\;t\in\mathbb R.
\end{eqnarray}
Denoting by $V_1,...,V_n$ the columns of the matrix
$(h_{lj}+t\delta_{lj})$, one has
\begin{eqnarray}\label{det4}
e_k(V_j)=\sum_lh_{ljk}E_l,\;\;\;j,k=1,...,n,
\end{eqnarray}
where $\{E_1,...,E_n\}$ is the canonical basis of $\mathbb R^n$.
Then, by (\ref{det4}) and multilinearity of the determinant,
\begin{eqnarray}\label{det5}
e_k\big(\text{det}\,(A+tI)\big)=
\sum_{j,l}h_{ljk}\text{det}\,(V_1,...,V_{j-1},E_l,V_{j+1},...,V_n).
\end{eqnarray}
Differentiating the above equality with respect to $e_k$ and using
(\ref{det4}), we obtain at $p$
\begin{eqnarray}\label{det7}
e_ke_k\big(\text{det}\,(A+tI)\big)&=&\sum_{i\neq
j}(h_{iik}h_{jjk}-h_{ijk}^2)\prod_{s\neq
i,j}(\lambda_s+t)\nonumber\\
&&+\sum_{j}h_{jjkk}\prod_{s\neq j}(\lambda_s+t).
\end{eqnarray}
By Lemma \ref{Inversion} and Codazzi equation, we have at $p$
\begin{eqnarray}\label{det8}
h_{jjkk}=h_{jkjk}&=&h_{jkkj}+\sum_m\mathcal R_{jkjm}h_{mk}+\sum_m\mathcal R_{jkkm}h_{jm}\nonumber\\
&=&h_{kkjj}+(\lambda_j-\lambda_k)\mathcal R_{jkkj}.
\end{eqnarray}
Covariant differentiation of the equality $nH=\sum_kh_{kk}$ gives
\begin{eqnarray}\label{det8a}
n\text{Hess}\,H(e_j,e_j)=\sum_{k}h_{kkjj},\;\;\;j=1,...,n.
\end{eqnarray}
Since the Laplacian of a function is the trace of its Hessian, and
$\nabla_{e_i}e_j(p)=0$, $1\leq i,j\leq n$, summing over $k$ in
(\ref{det7}), and using (\ref{det8}) and (\ref{det8a}), we arrive
at
\begin{eqnarray}\label{det10}
\Delta\,\text{det}\,(A+tI)&=&\sum_jnH_{jj}\prod_{s\neq j}(\lambda_s+t)
+\sum_{j\neq k}(\lambda_j-\lambda_k)\mathcal R_{jkkj}\prod_{s\neq j}(\lambda_s+t)\nonumber\\
&&+\sum_k\sum_{i\neq j}(h_{iik}h_{jjk}-h_{ijk}^2)\prod_{s\neq i,j}(\lambda_s+t),
\end{eqnarray}
where $H_{jj}=\text{Hess}\,H(e_j,e_j)$. Since
\begin{eqnarray}\label{det11}
\prod_{s\neq j}(\lambda_s+t)=[(\lambda_k-\lambda_j)+(\lambda_j+t)]\prod_{s\neq j,k}(\lambda_s+t),\nonumber
\end{eqnarray}
the second term on the right hand side of (\ref{det10}) can be written as
\begin{eqnarray}\label{det12}
\sum_{j\neq k}(\lambda_j-\lambda_k)\mathcal R_{jkkj}\prod_{s\neq j}(\lambda_s+t)
&=&-\sum_{j\neq k}(\lambda_j-\lambda_k)^2\mathcal R_{jkkj}\prod_{s\neq j,k}(\lambda_s+t)\nonumber\\
&&-\sum_{j\neq k}(\lambda_k-\lambda_j)\mathcal R_{jkkj}\prod_{s\neq k}(\lambda_s+t).\nonumber
\end{eqnarray}
Since $\mathcal R_{jkkj}=\mathcal R_{kjjk}$, the second term on
the right hand side of the above equality is minus the term on the
left. Hence
\begin{eqnarray}\label{det13}
\sum_{j\neq k}(\lambda_j-\lambda_k)\mathcal R_{jkkj}\prod_{s\neq j}(\lambda_s+t)=
-\sum_{j<k}(\lambda_j-\lambda_k)^2\mathcal R_{jkkj}\prod_{s\neq j,k}(\lambda_s+t).
\end{eqnarray}
It now follows from (\ref{SymmetricFunction1}), (\ref{SymmetricFunction2}), (\ref{det10}) and (\ref{det13}) that
\begin{eqnarray}\label{det14}
\Delta\,\text{det}\,(A+tI)&=&\sum_{r=0}^n\Big\{\sum_jnH_{jj}\frac{\partial\sigma_r}{\partial x_j}(\overrightarrow\lambda)
-\sum_{i<j}(\lambda_i-\lambda_j)^2\mathcal R_{ijji}\frac{\partial^2\sigma_r}{\partial x_i\partial x_j}(\overrightarrow\lambda)\nonumber\\
&&+\sum_{i,j,k}(h_{iik}h_{jjk}-h_{ijk}^2)\frac{\partial^2\sigma_r}{\partial x_i\partial x_j}(\overrightarrow\lambda)\Big\}t^{n-r}.
\end{eqnarray}
On the other hand, taking $x_s=\lambda_s$ in
(\ref{SymmetricFunction}) and using (\ref{DefinitionHr2}), one
obtains
\begin{eqnarray}\label{det5C}
\text{det}(A+tI)=\sum_{r=0}^n\binom{n}{r}H_rt^{n-r},
\end{eqnarray}
and so
\begin{eqnarray}\label{det16}
\Delta\,\text{det}\,(A+tI)=\sum_{r=0}^n\binom{n}{r}\Delta H_rt^{n-r}.
\end{eqnarray}
Comparing (\ref{det14}) and (\ref{det16}), one obtains
(\ref{WalterFormula}).
\end{proof}

\section{Complete and bounded hypersurfaces}

In the proof of Theorem \ref{MainResult} we will use, besides
Propositions \ref{Concave} and \ref{WalterProposition}, the
following result.

\begin{prop}\label{EllipticPoint}
Let $M^n$ be a complete Riemannian manifold with sectional
curvature $K$ bounded from below and $f:M^n\to\mathbb{Q}^{n+k}_c$
an isometric immersion of $M^n$ into the $(n+k)$-dimensional
complete simply-connected space $\mathbb{Q}^{n+k}_c$ of constant
sectional curvature $c$. In the case $c\leq 0$, assume that $f(M^n)$ is bounded, and in the case $c>0$, that $f(M^n)$ lies inside a
geodesic ball of radius $\rho<\pi/2\sqrt{c}$.  Then, there exist $p\in M$ and a unit vector $\xi_0\in(f_{\ast}T_p M)^{\bot}$ such that, for
any unit vector $v\in T_p M$,
\begin{eqnarray}\label{ast_propo_elipt}
\langle A_{\xi_0} v,v\rangle>\left\{
\begin{array}{ll}
0,&c\geq 0\\
\sqrt{-c},&c< 0.
\end{array}
\right.
\end{eqnarray}
\end{prop}

We believe that the above proposition is known, but since we were
unable to find a reference for it in the literature, we will prove
it below. The main ingredient in this proof is the following well
known maximum principle due to Omori and Yau \cite{CY1,Om,Y1} (see
\cite[Theorem 3.4]{FX} for a conceptual refinement of this
principle):

\vskip10pt

\noindent{\bf Omori-Yau Maximum Principle.} {\em Let $M^n$ be a
complete Riemannian manifold with sectional curvature (resp. Ricci
curvature) bounded from below, and let $f:M\to\mathbb{R}$ be a
$C^2$-function bounded from above. Then, for every
$\varepsilon>0$, there exists $x_{\varepsilon}\in M$ such that}
\begin{eqnarray}\label{SpecialSequence1}
f(x_{\varepsilon})>\sup f-\varepsilon,\,||\nabla
f(x_{\varepsilon})||<\varepsilon,\,\text{Hess}
f(x_{\varepsilon})(v,v)<\varepsilon\|v\|^2\,\big(\emph{resp.}\,
\Delta f(x_{\varepsilon})<\varepsilon\big).\nonumber
\end{eqnarray}

The following lemma, which will also be used in the proof of Proposition \ref{EllipticPoint}, expresses the gradient and Hessian of the
restriction of a function to a submanifold in terms of the space
gradient and Hessian (see \cite[p. 46]{Da} for a proof). In its
statement, we will use the symbol $\nabla$ for the gradient of any
function involved.
\begin{lem}\label{HessianComposition}
Let $f:M^n\to N^{n+k}$ be an isometric immersion of a Riemannian
manifold $M^n$ into a Riemannian manifold $N^{n+k}$, and let
$g:N\to\mathbb R$ be a function of class $C^2$. Then, for all
$p\in M$ and $v,w\in T_{p}M$, one has
\begin{eqnarray}
f_{\ast}\big(\nabla (g\circ f)(p)\big)=\Big[\nabla
g(f(p))\Big]^{\top},
\end{eqnarray}
\vskip-15pt
\begin{eqnarray}\label{HessianCompositionA}
\emph{Hess}\,(g\circ
f)_p(v,w)=\emph{Hess}\,g_{f(p)}(f_{\ast}v,f_{\ast}w)+\Big\langle\nabla
g(f(p)),\sigma_p(v,w)\Big\rangle,
\end{eqnarray}
where $\sigma$ is the second fundamental form of the immersion, $f_{\ast}$
is the differential of $f$ and ``$\top$'' means orthogonal
projection onto $f_{\ast}(T_p M)$.
\end{lem}

\noindent{\it Proof of Proposition \ref{EllipticPoint}.} By hypothesis, $f(M)$ is contained in some closed ball $\overline
B_{\rho}(q_o)$ of center $q_o$ and radius $\rho$, with
$\rho<\pi/2\sqrt{c}$ if $c>0$. Let $r(\cdot)=d(\cdot,q_0)$ be the distance
function from the point $q_0$ in $\mathbb{Q}^{n+k}_c$ and let
$g=r\circ f$. Since $g$ is bounded from above (for
$f(M)\subset\overline{B}_{\rho}(q_0)$) and the
sectional curvatures of $M$ are bounded from below, the Omori-Yau
maximum principle assures us that, for every $\varepsilon>0$, there
exist $x_{\varepsilon}\in M$ such that
\begin{eqnarray}\label{eli101}
g(x_{\varepsilon})>\sup g-\varepsilon,\;\;\|\nabla
g(x_{\varepsilon})\|<\varepsilon,\;\;
\text{Hess}g_{x_{\varepsilon}}(v,v)<\varepsilon \|v\|^2,\;\forall
v\in T_{x_{\varepsilon}} M.\nonumber
\end{eqnarray}
From the last two inequalities and Lemma \ref{HessianComposition},
we obtain
\begin{eqnarray}\label{eli102}
\varepsilon>\|\nabla g(x_{\varepsilon})\|=\|\nabla
r(f(x_{\varepsilon}))^{\top}\|
\end{eqnarray}
and, for every $v\in T_{x_{\varepsilon}} M$,
\begin{eqnarray}\label{eli103}
\varepsilon
\|v\|^2>\text{Hess}g_{x_{\varepsilon}}(v,v)=\text{Hess}\;r_{f(x_{\varepsilon})}
(f_{\ast}v,f_{\ast}v)+\Big\langle\sigma_{x_{\varepsilon}}(v,v),\nabla
r(f(x_{\varepsilon}))\Big\rangle,
\end{eqnarray}
where the superscript ``$\top$'' indicates orthogonal projection
on $f_{\ast}(T_{x_{\varepsilon}} M)$.

For every $v\in T_{x_{\varepsilon}} M$, write
\begin{eqnarray}\label{eli104}
f_{\ast}v=v_1+v_2,
\end{eqnarray}
where $v_1$ and $v_2$ are the components of $f_{\ast}v$ that are parallel and orthogonal,
respectively, to $\nabla
r(f(x_{\varepsilon}))$. Recalling that $\overline{\nabla}_{\nabla r}\nabla
r=0$, where $\overline{\nabla}$ is the Riemannian connection of
$\mathbb{Q}^{n+k}_c$, one
has
\begin{eqnarray}\label{eli105}
\text{Hess}\,r_{f(x_{\varepsilon})}(f_{\ast} v,f_{\ast}v)
&=&\text{Hess}\,r_{f(x_{\varepsilon})}(v_1+v_2,v_1+v_2)\nonumber\\&=&\text{Hess}\;r_{f(x_{\varepsilon})}(v_2,v_2).
\end{eqnarray}
Note that $v_2$ is tangent to the geodesic sphere $S$ of $\mathbb
Q^{n+k}_c$ centered at $q_0$ that contains $f(x_{\epsilon})$.
Applying (\ref{HessianCompositionA}) for the inclusion
$\iota:S\to\mathbb Q^{n+k}_c$ and $g=r$, one obtains
\begin{eqnarray}\label{eli105A}
\text{Hess}\,r_{f(x_{\varepsilon})}(v_2,v_2)=\langle Bv_2,v_2\rangle,
\end{eqnarray}
where $B$ is the shape operator of $S$ with respect to $-\nabla
r$. Since the principal curvatures of a geodesic sphere of radius
$t$ in $\mathbb Q^{n+k}_c$ are constant and given by
\begin{eqnarray}\label{eli106}
\mu_c(t)=\left\{
\begin{array}{ll}
\sqrt{c}\;\text{cot}(\sqrt{c}\;t),&c>0,\;0<t<{\pi}/{\sqrt{c}},\\
1/t, &c=0,\;t>0,\\
\sqrt{-c}\;\text{coth}(\sqrt{-c}\;t),&c<0,\;t>0,
\end{array}
\right.
\end{eqnarray}
it follows from (\ref{eli105}) and (\ref{eli105A}) that
\begin{eqnarray}\label{eli106A}
\text{Hess}\,r_{f(x_{\varepsilon})}(f_{\ast} v,f_{\ast}v)=\mu_c(r(f(x_{\varepsilon})))||v_2||^2.
\end{eqnarray}
As $\|\nabla r\|\equiv 1$, by (\ref{eli104}) one has $v_1=\langle
f_{\ast}v,\nabla r(f(x_{\varepsilon}))\rangle\nabla
r(f(x_{\varepsilon}))$. Then, by (\ref{eli102}),
\begin{eqnarray}\label{eli107}
\|v_1\|=|\langle f_{\ast}v,\nabla
r(f(x_{\varepsilon}))^{\top}\rangle|
\leq\big\|f_{\ast}v\big\|\big\|\nabla
r(f(x_{\varepsilon}))^{\top}\big\|<\varepsilon\|v\|.
\end{eqnarray}
From (\ref{eli104}) and (\ref{eli107}), we obtain
\begin{eqnarray}\label{eli108}
\|v_2\|^2=\|f_{\ast}v\|^2-\|v_1\|^2=\|v\|^2-\|v_1\|^2>(1-\varepsilon^2)\|v\|^2.
\end{eqnarray}
Hence, by (\ref{eli103}), (\ref{eli106A}) and (\ref{eli108}),
\begin{eqnarray*}
\varepsilon \|v\|^2>
\mu_c(r(f(x_{\varepsilon})))(1-\varepsilon^2)\|v\|^2+\big\langle\sigma_{x_{\varepsilon}}(v,v),\nabla
r(f(x_{\varepsilon}))\big\rangle.
\end{eqnarray*}
Since $\mu_c$ is decreasing and $r(f(x_{\varepsilon}))\leq\rho$,
it follows that
\begin{eqnarray*}
\varepsilon
\|v\|^2&>&\mu_c(\rho)(1-\varepsilon^2)\|v\|^2+\big\langle\sigma_{x_{\varepsilon}}(v,v),\nabla r(f(x_{\varepsilon}))\big\rangle\nonumber\\
&=&\mu_c(\rho)(1-\varepsilon^2)\|v\|^2+\big\langle\sigma_{x_{\varepsilon}}(v,v),\nabla
r(f(x_{\varepsilon}))^{\perp}\big\rangle,
\end{eqnarray*}
where $\nabla r(f(x_{\varepsilon}))^{\perp}$ is the component of
$\nabla r(f(x_{\varepsilon}))$ that is orthogonal to $f_{\ast}(T_{x_{\varepsilon}}
M)$. Setting $\xi_{\varepsilon}=-\nabla
r(f(x_{\varepsilon}))^{\perp}/||\nabla
r(f(x_{\varepsilon}))^{\perp}||$, it follows from (\ref{Sigma})
and the above inequality that
\begin{eqnarray}\label{eli109}
\langle
A_{\xi_{\varepsilon}}v,v\rangle=\langle\sigma_{x_{\varepsilon}}(v,v),\xi_{\varepsilon}\rangle
>\frac{\mu_c(\rho)(1-\varepsilon^2)-\varepsilon}{\|\nabla
r(f(x_{\varepsilon}))^{\perp}\|}\,,
\end{eqnarray}
for all $v\in T_{x_{\varepsilon}}M,\;||v||=1$. Since, by
(\ref{eli102}), the term on the right hand side of (\ref{eli109})
tends to $\mu_c(\rho)$ when $\varepsilon\to 0$, and, by
(\ref{eli106}), $\mu_c(\rho)>0$ for $c\geq 0$ and
$\mu_c(\rho)>\sqrt{-c}$ for $c<0$, (\ref{ast_propo_elipt}) is
fulfilled choosing $p=x_{\varepsilon}$ and
$\xi_0=\xi_{\varepsilon}$, where $\varepsilon$ is any positive
number sufficiently small.\qed

\section{Proof of Theorem \ref{MainResult}.}

Since $H$ is constant and $R$ is bounded from below, from
(\ref{Relation}) one obtains that $|A|^2$ is bounded, and so that
the sectional curvatures of $M^n$ are bounded from below. Then, by  Proposition \ref{EllipticPoint}, 
there exist a point $p\in M$ and a unit vector
$\xi_0\in(f_{\ast} T_{p}M)^{\bot}$ such that
\begin{eqnarray}\label{Proof1}
\langle A_{\xi_0}v,v\rangle>\alpha_c ||v||^2,\;\;\;v\in T_{p} M,
\end{eqnarray}
where
\begin{eqnarray}\label{Proof1a}
\alpha_c=\begin{cases}
0, &c\geq0,\\
\sqrt{-c},&c<0.
\end{cases}
\end{eqnarray}
Choosing the unit normal vector field $\xi$ such that
$\xi(p)=\xi_0$, by (\ref{Proof1}) the principal curvatures of $M$
at $p$ satisfy
\begin{eqnarray}\label{Proof2}
\lambda_i(p)>\alpha_c\geq 0,\;\;\;i=1,\cdots,n.
\end{eqnarray}

By Proposition \ref{WalterProposition} one has, as $H$ and $H_r$
are constant,
\begin{eqnarray}\label{Proof5}
\sum_{i<j}\frac{\partial^2\sigma_r}{\partial x_i\partial
x_j}(\overrightarrow{\lambda})(\lambda_i-\lambda_j)^2K_{ij}
=\sum_{i,j,k}\frac{\partial^2\sigma_r}{\partial x_i\partial
x_j}(\overrightarrow{\lambda})(h_{iik}h_{jjk}-h_{ijk}^2),
\end{eqnarray}
where $\overrightarrow{\lambda}=(\lambda_1,...,\lambda_n)$. From
(\ref{Proof2}) one obtains that $H_r>0$ and that
$\overrightarrow{\lambda}(p)$ belongs to the G\aa rding's cone
$\Gamma_r$ (see Section 2). Then, since $M$ is connected,
$\overrightarrow{\lambda}(q)\in\Gamma_r,\;\forall q\in M$.

By Proposition \ref{Concave}, $W_r=\sigma_r^{1/r}$ is a concave
function on $\Gamma_r$. Thus,
\begin{eqnarray}\label{Proof8}
\sum_{i,j}y_iy_j\frac{\partial^2 W_r}{\partial x_i\partial
x_j}(x)\leq 0,
\end{eqnarray}
for all $x\in\Gamma_r$ and $(y_1,\ldots,y_n)\in\mathbb R^n$. A
simple computation shows that
\begin{eqnarray}\label{Proof10}
\frac{\partial^2 W_r}{\partial x_i\partial x_j}
=\frac{1}{r}\sigma_r^{\frac{1}{r}-2}\left(\frac{1-r}{r}\frac{\partial\sigma_r}{\partial
x_i}\frac{\partial\sigma_r}{\partial x_j}
+\sigma_r\frac{\partial^2\sigma_r}{\partial x_i\partial
x_j}\right).
\end{eqnarray}
Using (\ref{Proof10}) in (\ref{Proof8}), we conclude that
\begin{eqnarray}\label{Proof11}
\sigma_r(x)\sum_{i,j}y_{i}y_{j}\frac{\partial^2\sigma_r}{\partial
x_i\partial x_j}(x)
&\leq&\frac{r-1}{r}\sum_{i,j}y_{i}y_{j}\frac{\partial\sigma_r}{\partial
x_i}(x)\frac{\partial\sigma_r}{\partial
x_j}(x)\nonumber\\&=&\frac{r-1}{r}\left(\sum_jy_j\frac{\partial\sigma_r}{\partial
x_j}(x)\right)^2,
\end{eqnarray}
for all $x\in\Gamma_r$ and $(y_1,\cdots,y_n)\in\mathbb R^n$.
Taking $x=\overrightarrow{\lambda}$ and $y_i=h_{iik}$,
$i=1,\ldots,n$, in (\ref{Proof11}), one obtains
\begin{eqnarray}\label{Proof12}
\binom{n}{r}H_r\sum_{i,j}h_{iik}h_{jjk}\frac{\partial^2\sigma_r}{\partial
x_i\partial x_j}(\overrightarrow{\lambda})
\leq\frac{r-1}{r}\Big(\sum_jh_{jjk}\frac{\partial\sigma_r}{\partial
x_j}(\overrightarrow{\lambda})\Big)^2,\;\;\;\forall k.
\end{eqnarray}

We claim that in a basis that diagonalizes $A$,
\begin{eqnarray}\label{det5E}
\sum_{j}h_{jjk}\frac{\partial\sigma_r}{\partial
x_j}(\overrightarrow{\lambda})=\binom{n}{r}e_k(H_r).
\end{eqnarray}
The claim can be proved using the formula \cite[p. 225]{Ro}
\begin{eqnarray}\label{GradientH_r}
\binom{n}{r}e_k(H_r)=\text{trace}\big(P_{r-1}\nabla_{e_k}A\big),\nonumber
\end{eqnarray}
where $P_{r-1}$ is the $(r-1)$-th Newton tensor associated with
the shape operator $A$ of $M$. Alternatively, (\ref{det5E}) can be
obtained from the computations made in the proof of Proposition
\ref{WalterProposition}. In fact, by (\ref{SymmetricFunction1})
and (\ref{det5}) we have
\begin{eqnarray}\label{det5A}
e_k\big(\text{det}\,(A+tI)\big)&=&\sum_{j}h_{jjk}\prod_{s\neq
j}(\lambda_s+t).\nonumber\\&=&\sum_{r=0}^n\left(\sum_{j}h_{jjk}\frac{\partial\sigma_r}{\partial
x_j}(\overrightarrow{\lambda})\right)t^{n-r}.
\end{eqnarray}
On the other hand, by (\ref{det5C}) one has
\begin{eqnarray}\label{det5D}
e_k\big(\text{det}\,(A+tI)\big)=\sum_{r=0}^n\binom{n}{r}e_k(H_r)t^{n-r}.
\end{eqnarray}
Comparing (\ref{det5A}) and (\ref{det5D}), one obtains (\ref{det5E}).

Since $H_r$ is a
positive constant, from (\ref{Proof12}) and (\ref{det5E}) one obtains
\begin{eqnarray}\label{Proof16}
\sum_{i,j}h_{iik}h_{jjk}\frac{\partial^2\sigma_r}{\partial
x_i\partial x_j}(\overrightarrow{\lambda})
\leq0,\;\;\;k=1,...,n.\nonumber
\end{eqnarray}
Using this information in (\ref{Proof5}), we conclude that the
inequality
\begin{eqnarray}\label{Proof5A}
\sum_{i<j}\frac{\partial^2\sigma_r}{\partial x_i\partial
x_j}(\overrightarrow{\lambda})(\lambda_i-\lambda_j)^2K_{ij}
\leq-\sum_{i,j,k}h_{ijk}^2\frac{\partial^2\sigma_r}{\partial
x_i\partial x_j}(\overrightarrow{\lambda})\nonumber
\end{eqnarray}
holds at every point of $M$. Since, by (\ref{SymmetricFunction5})
and (\ref{Proof2}),
\begin{eqnarray}\label{Proof19}
\frac{\partial^2\sigma_r}{\partial x_i\partial
x_j}(\overrightarrow{\lambda}(p))=\begin{cases}\sigma_{r-2}(\widehat{\lambda_i}(p),\widehat{\lambda_j}(p))>0,\;\;\;&i\neq
j,\\0,\;\;\;&i=j,\end{cases}
\end{eqnarray}
it follows that
\begin{eqnarray}\label{Proof17}
\sum_{i<j}\frac{\partial^2\sigma_r}{\partial x_i\partial
x_j}(\overrightarrow{\lambda}(p))(\lambda_i(p)-\lambda_j(p))^2K_{ij}(p)\leq0.
\end{eqnarray}
Since, by (\ref{Proof2}) and the Gauss equation,
\begin{eqnarray}\label{Proof18}
K_{ij}(p)=c+\lambda_i(p)\lambda_j(p)>c+\alpha_c^2\geq0,\;\;\;i\neq
j,\nonumber
\end{eqnarray}
it follows from (\ref{Proof19}) and (\ref{Proof17}) that
\begin{eqnarray}\label{Proof20}
\lambda_1(p)=\cdots=\lambda_n(p)=H.
\end{eqnarray}

The above argument in fact shows that every point $q\in M$ for
which $\lambda_i(q)>\alpha_c,\;\forall i,$ is umbilical. Since
$H>\alpha_c$ by (\ref{Proof2}) and (\ref{Proof20}), one then has
that the set $B$ of all the umbilical points of $M$ is open. Since
$B$ is also nonempty (for $p\in M$) and closed (by the continuity
of the principal curvature functions), one concludes that $B=M$
from the connectedness of $M$. Hence,
\begin{eqnarray}\label{Proof21}
\lambda_1=\cdots=\lambda_n=H>\alpha_c,
\end{eqnarray}
at any point of $M$. It now follows from (\ref{Proof1a}),
(\ref{Proof21}) and the Gauss equation that the sectional
curvature of $M$ satisfies $K=c+H^2>0$. In particular, $M$ is
compact. It now follows from the classification of the umbilical
hypersurfaces in a simply connected space form (see, for instance,
\cite[p. 25]{BCO}) that $f(M)$ is a hypersphere of $\mathbb
Q^{n+1}_c$.\qed

$$
\begin{array}{lccl}
\text{Francisco Fontenele}            && & \text{Roberto Alonso N\'u\~nez}\\
\text{Departamento de Geometria}      && & \text{Rua M\'ario Santos Braga s/n}\\
\text{Universidade Federal Fluminense}&& & \text{24020-140\;\;Niter\'oi, RJ, Brazil}\\
\text{Niter\'oi, RJ, Brazil}          && & \text{\detokenize{roberto78nunez@gmail.com}}\\
\texttt{\detokenize{fontenele@mat.uff.br}}           && & \texttt{}\\
\end{array}
$$

\end{document}